\documentclass[11pt]{article}
\usepackage{latexsym}
\usepackage{amsfonts}
\usepackage{mathrsfs,amsthm,amsmath}
\usepackage{color}
\usepackage{todonotes}
\newtheorem{theorem}{Theorem}[section]
\newtheorem{proposition}[theorem]{Proposition}

\newtheorem{lemma}{Lemma}[section]

\newtheorem{remark}{Remark}[section]
\newtheorem{example}{Example}[section]
\newtheorem{condition}{Condition}[section]
\begin{document}
\title{Noncentral moderate deviations for time-changed Lévy processes with inverse of stable 
subordinators}
\author{Antonella Iuliano\thanks{Address: Dipartimento di Matematica, Informatica ed Economia, 
Università degli Studi della Basilicata, Via dell'Ateneo Lucano 10 (Campus di Macchia Romana), 
I-85100 Potenza, Italy. e-mail: \texttt{antonella.iuliano@unibas.it}}
\and Claudio Macci\thanks{Address: Dipartimento di Matematica, Università di Roma Tor Vergata,
Via della Ricerca Scientifica, I-00133 Rome, Italy. e-mail: \texttt{macci@mat.uniroma2.it}}
\and Alessandra Meoli\thanks{Dipartimento di Matematica, Università degli Studi di Salerno, 
Via Giovanni Paolo II n. 132, 84084 Fisciano, SA, Italy. e-mail: \texttt{ameoli@unisa.it}}}
\maketitle
\begin{abstract}
In this paper we present some extensions of recent noncentral moderate deviation results in the literature.
In the first part we generalize the results in \cite{BeghinMacciSPL2022} by considering a general Lévy process
$\{S(t):t\geq 0\}$ instead of a compound Poisson process. In the second part we assume that $\{S(t):t\geq 0\}$
has bounded variation and is not a subordinator; thus $\{S(t):t\geq 0\}$ can be seen as the difference of two 
independent non-null subordinators. In this way we generalize the results in \cite{LeeMacci} for Skellam processes.
\end{abstract}

\section{Introduction}
The theory of large deviations gives an asymptotic computation of small probabilities on exponential scale (see 
\cite{DemboZeitouni} as a reference of this topic). In particular a \emph{large deviation principle} (LDP from now on) 
provides asymptotic bounds for families of probability measures on the same topological space; these bounds are expressed 
in terms of a \emph{speed function} (that tends to infinity) and a nonnegative lower semicontinuous \emph{rate function} 
defined on the topological space.

The term \emph{moderate deviations} is used for a class of LDPs which fills the gap between the two following asymptotic 
regimes: a convergence to a constant $x_0$ (at least in probability), and a weak convergence to a non-constant centered 
Gaussian random variable. The convergence to $x_0$ is governed by a \emph{reference LDP} with speed $v_t$, and a rate 
function that uniquely vanishes at $x_0$. Then we have a class of LDPs which depends on the choice of some positive 
scalings $\{a_t:t>0\}$ such that $a_t\to 0$ and $a_tv_t\to\infty$ (as $t\to\infty$).

In this paper the topological space mentioned above is the real line $\mathbb{R}$ equipped with the Borel 
$\sigma$-algebra, $x_0=0\in\mathbb{R}$ and $v_t=t$; thus we shall have
\begin{equation}\label{eq:MDconditions}
	a_t\to 0\ \mbox{and}\ a_tt\to\infty\quad (\mbox{as}\ t\to\infty).
\end{equation}

In some recent papers (see e.g. \cite{GiulianoMacci} and some references cited therein), the term
\emph{noncentral moderate deviations} has been introduced when one has the situation described above, but the weak 
convergence is towards a non-constant and non-Gaussian distributed random variable. A multivariate
example is given in \cite{LeonenkoMacciPacchiarotti}.

A possible way to construct examples of moderate deviation results
is the following: $\{C_t:t>0\}$ is a family of random variables which converges to zero as $t\to\infty$ (and satisfies 
the reference LDP with a certain rate function and a certain speed $v_t$); moreover, for a certain function $\phi$ such 
that $\phi(x)\to\infty$ as $x\to\infty$, $\{\phi(v_t)C_t:t>0\}$ converges weakly to some non-degenerating random variable; 
then, for every sequence family of positive scalings $\{a_t:t>0\}$ such that $a_t\to 0$ and $a_tv_t\to\infty$ 
(as $t\to\infty$), one should be able to prove the LDP for $\{\phi(a_tv_t)C_t:t>0\}$ with a certain rate function and 
speed $1/a_t$. We remark that, according to this approach, one typically has $\phi(x)=\sqrt{x}$ for central moderate deviation
results (i.e. for the cases in which the weak convergence is towards a Normal distribution); moreover, to give an example with
a different situation, we have $\phi(x)=x$ for the noncentral moderate deviation result in \cite{IafrateMacci} (note that 
$r$ and $\gamma_r$ in that reference plays the role of $t$ and $a_t$ in this paper). The results in this paper follow this
approach with $\phi(x)=x^\beta$ for some $\beta\in(0,1)$; more precisely $\beta=\alpha(\nu)$ (see \eqref{eq:exponents}) in 
Section \ref{sec:gen-results-Beghin-Macci-2022} and $\beta=\alpha_1(\nu)$ (see \eqref{eq:exponent}) in Section 
\ref{sec:gen-results-Lee-Macci}.

Our aim is to present some extensions of the recent results presented in \cite{BeghinMacciSPL2022} and \cite{LeeMacci}.
In particular we recall that a subordinator is a nondecreasing (real-valued) Lévy process. Throughout this paper we 
always deal with real-valued light-tailed Lévy processes $\{S(t):t\geq 0\}$ described in the next Condition 
\ref{cond:real-lighttailed-Levyprocess}, with an independent random time-change in terms of inverse of stable subordinators.

\begin{condition}\label{cond:real-lighttailed-Levyprocess}
	Let $\{S(t):t\geq 0\}$ be a real-valued Lévy process, and let $\kappa_S$ be the function defined by
	$$\kappa_S(\theta):=\log\mathbb{E}[e^{\theta S(1)}].$$
	We assume that the function $\kappa_S$ is finite in a neighborhood of the origin $\theta=0$. 
	In particular the random variable $S(1)$ has finite mean $m:=\kappa_S^\prime(0)$ and finite variance 
	$q:=\kappa_S^{\prime\prime}(0)$.
\end{condition}

We recall that, if $\{S(t):t\geq 0\}$ in Condition \ref{cond:real-lighttailed-Levyprocess} is a Poisson process and
$\{L_\nu(t):t\geq 0\}$ is an independent inverse of stable subordinators, then the process $\{S(L_\nu(t)):t\geq 0\}$ is a 
(time) fractional Poisson process  (see \cite{MeerschaertNaneVellaisamy}; see also Section 2.4 in \cite{MeerschaertSikorskii} 
for more general time fractional processes).

The aim of this paper is to provide some extensions of recent noncentral moderate deviation results in the literature. More 
precisely we mean:
\begin{enumerate}
	\item the generalization of the results in \cite{BeghinMacciSPL2022} by considering a general Lévy process
	$\{S(t):t\geq 0\}$ instead of a compound Poisson process;
	\item the generalization of the results in \cite{LeeMacci} by considering the difference between two non-null
	independent subordinators $\{S(t):t\geq 0\}$ instead of a Skellam process (which is the difference between two 
	independent Poisson processes).
\end{enumerate}
For the first item we have only one (independent) random time-change for $\{S(t):t\geq 0\}$, and we can specify the results to 
the fractional Skellam processes of type 2 in \cite{KerssLeonenkoSikorskii}. For the second item we shall assume that 
$\{S(t):t\geq 0\}$ has bounded variation and is not a subordinator; thus $\{S(t):t\geq 0\}$ can be seen as the difference of 
two independent non-null subordinators $\{S_1(t):t\geq 0\}$ and $\{S_2(t):t\geq 0\}$ (see Lemma \ref{lem:Chi} in this paper).
Then, for the second item, we have two (independent) random time-changes for $\{S_1(t):t\geq 0\}$ and $\{S_2(t):t\geq 0\}$, and 
we can specify the results to the fractional Skellam processes of type 1 in \cite{KerssLeonenkoSikorskii}.

The outline of the paper is as follows. In Section \ref{sec:preliminaries} we recall some preliminaries. The extensions presented
above in items 1 and 2 are studied in Sections \ref{sec:gen-results-Beghin-Macci-2022} and \ref{sec:gen-results-Lee-Macci}, 
respectively. In Section \ref{sec:referee} we discuss the possibility to have some generalizations with more general random 
time-changes, and in particular we present Propositions \ref{prop:LD-gen} and \ref{prop:LD-bis-gen} which provide a generalization
of the reference LDPs in Propositions \ref{prop:LD} and \ref{prop:LD-bis}, respectively.
In Section \ref{sec:comparisons} we present some comparisons between rate functions, and in particular we follow the 
same lines of the comparisons in Section 5 in \cite{LeeMacci}. Finally, motivated by potential applications to other fractional 
processes in the literature, in Section \ref{sec:diff-TSS} we discuss the case of the difference of two (independent) tempered 
stable subordinators.

\section{Preliminaries}\label{sec:preliminaries}
In this section we recall some preliminaries on large deviations and on the inverse of the stable subordinator,
together with the Mittag-Leffler function.

\subsection{Preliminaries on large deviations}
We start with some basic definitions (see e.g. \cite{DemboZeitouni}). In view of what follows we present definitions and results
for families of real random variables $\{Z(t):t>0\}$ defined on the same probability space $(\Omega,\mathcal{F},P)$, where $t$ 
goes to infinity. A real-valued function $\{v_t:t>0\}$ such that $v_t\to\infty$ (as $t\to\infty$) is called a 
\emph{speed function}, and a lower semicontinuous function $I:\mathbb{R}\to[0,\infty]$ is called a \emph{rate function}. Then 
$\{Z(t):t>0\}$ satisfies the LDP with speed $v_t$ and a rate function $I$ if
$$\limsup_{t\to\infty}\frac{1}{v_t}\log P(Z(t)\in C)\leq-\inf_{x\in C}I(x)\quad\mbox{for all closed sets}\ C,$$
and
$$\liminf_{t\to\infty}\frac{1}{v_t}\log P(Z(t)\in O)\geq-\inf_{x\in O}I(x)\quad\mbox{for all open sets}\ O.$$
The rate function $I$ is said to be \emph{good} if, for every $\beta\geq 0$, the level set 
$\{x\in\mathbb{R}:I(x)\leq\beta\}$ is compact. We also recall the following known result (see e.g. Theorem
2.3.6(c) in \cite{DemboZeitouni}).

\begin{theorem}[G\"artner Ellis Theorem]\label{th:GE}
Assume that, for all $\theta\in\mathbb{R}$, there exists
$$\Lambda(\theta):=\lim_{t\to\infty}\frac{1}{v_t}\log\mathbb{E}\left[e^{v_t\theta Z(t)}\right]$$
as an extended real number; moreover assume that the origin $\theta=0$ 
belongs to the interior of the set $\mathcal{D}(\Lambda):=\{\theta\in\mathbb{R}:\Lambda(\theta)<\infty\}$.
Furthermore let $\Lambda^*$ be the Legendre-Fenchel transform of $\Lambda$, i.e. the function defined by
$$\Lambda^*(x):=\sup_{\theta\in\mathbb{R}}\{\theta x-\Lambda(\theta)\}.$$
Then, if $\Lambda$ is essentially smooth and lower semi-continuous, then $\{Z(t):t>0\}$
satisfies the LDP with good rate function $\Lambda^*$.
\end{theorem}

We also recall (see e.g. Definition 2.3.5 in \cite{DemboZeitouni}) that $\Lambda$ is essentially smooth
if the interior of $\mathcal{D}(\Lambda)$ is non-empty, the function $\Lambda$ is differentiable 
throughout the interior of $\mathcal{D}(\Lambda)$, and $\Lambda$ is steep, i.e. $|\Lambda^\prime(\theta_n)|\to\infty$
whenever $\theta_n$ is a sequence of points in the interior of $\mathcal{D}(\Lambda)$ which converge to 
a boundary point of $\mathcal{D}(\Lambda)$.

\subsection{Preliminaries on the inverse of a stable subordinator}
We start with the definition of the Mittag-Leffler function (see e.g. \cite{GorenfloKilbasMainardiRogosin}, eq. 
(3.1.1))
$$E_\nu(x):=\sum_{k=0}^\infty\frac{x^k}{\Gamma(\nu k+1)}.$$
It is known (see Proposition 3.6 in \cite{GorenfloKilbasMainardiRogosin} for the case $\alpha\in(0,2)$; indeed 
$\alpha$ in that reference coincides with $\nu$ in this paper) that we have
\begin{equation}\label{eq:ML-asymptotics}
\left\{\begin{array}{l}
E_\nu(x)\sim\frac{e^{x^{1/\nu}}}{\nu}\ \mbox{as}\ x\to\infty\\
\mbox{for $y<0$, we have} \frac{1}{x}\log E_\nu(yx)\to 0\ \mbox{as}\ x\to\infty.
\end{array}\right.
\end{equation}

Then, if we consider the inverse of the stable subordinator $\{L_\nu(t):t\geq 0\}$ for $\nu\in(0,1)$, we have 
\begin{equation}\label{eq:MGF-inverse-stable-sub}
\mathbb{E}[e^{\theta L_\nu(t)}]=E_\nu(\theta t^\nu)\ \mbox{for all}\ \theta\in\mathbb{R}.
\end{equation}
This formula appears in several references with $\theta\leq 0$ only; however this restriction is not
needed because we can refer to the analytic continuation of the Laplace transform with complex argument.

\section{Results with only one random time-change}\label{sec:gen-results-Beghin-Macci-2022}
Throughout this section we assume that the following condition holds.

\begin{condition}\label{cond:gen-BM}
	Let $\{S(t):t\geq 0\}$ be a real-valued Lévy process as in Condition \ref{cond:real-lighttailed-Levyprocess},
	and let $\{L_\nu(t):t\geq 0\}$ be an inverse of a stable subordinator for $\nu\in(0,1)$. Moreover assume that
	$\{S(t):t\geq 0\}$ and $\{L_\nu(t):t\geq 0\}$ are independent.
\end{condition}

The next Propositions \ref{prop:LD}, \ref{prop:weak-convergence} and \ref{prop:ncMD} provide a generalization of 
Propositions 3.1, 3.2 and 3.3 in \cite{BeghinMacciSPL2022}, respectively, in which $\{S(t):t\geq 0\}$ is a compound
Poisson process. We start with the reference LDP for the convergence in probability to zero of 
$\left\{\frac{S(L_\nu(t))}{t}:t>0\right\}$.

\begin{proposition}\label{prop:LD}
Assume that Condition \ref{cond:gen-BM} holds. Moreover let $\Lambda_{\nu,S}$
be the function defined by
\begin{equation}\label{eq:LD-GE-limit}
\Lambda_{\nu,S}(\theta):=\left\{\begin{array}{ll}
(\kappa_S(\theta))^{1/\nu}&\ \mbox{if}\ \kappa_S(\theta)\geq 0\\
0&\ \mbox{if}\ \kappa_S(\theta)<0,
\end{array}\right.
\end{equation}
and assume that it is an essentially smooth function. Then $\left\{\frac{S(L_\nu(t))}{t}:t>0\right\}$ 
satisfies the LDP with speed $v_t=t$ and good rate function $I_{\mathrm{LD}}$ defined by
\begin{equation}\label{eq:LD-rf}
I_{\mathrm{LD}}(x):=\sup_{\theta\in\mathbb{R}}\{\theta x-\Lambda_{\nu,S}(\theta)\}.
\end{equation}
\end{proposition}
\begin{proof}
The desired LDP can be derived by applying the G\"artner Ellis Theorem (i.e. Theorem \ref{th:GE}). In fact we have
\begin{equation}\label{eq:MGF-S}
	\mathbb{E}[e^{\theta S(L_\nu(t))}]=E_\nu(\kappa_S(\theta)t^\nu)\quad \mbox{for all $\theta\in\mathbb{R}$ and for all $t\geq 0$},
\end{equation}
whence we obtain
$$\lim_{t\to\infty}\frac{1}{t}\log\mathbb{E}[e^{\theta S(L_\nu(t))}]=\Lambda_{\nu,S}(\theta)
\ \mbox{for all}\ \theta\in\mathbb{R}$$
by \eqref{eq:ML-asymptotics}.
\end{proof}

\begin{remark}\label{rem:es-not-guaranteed}
	The function $\Lambda_{\nu,S}$ in Proposition \ref{prop:LD}, eq. \eqref{eq:LD-GE-limit}, could not be essentially smooth.
	Here we present a counterexample. Let $\{S(t):t\geq 0\}$ be defined by $S(t):=S_1(t)-S_2(t)$, where $\{S_1(t):t\geq 0\}$
	is a tempered stable subordinator with parameters $\beta\in(0,1)$ and $r>0$, and let $\{S_2(t):t\geq 0\}$ be the 
	deterministic subordinator defined by $S_2(t)=ht$ for some $h>0$. Then
    $$\kappa_{S_1}(\theta):=\left\{\begin{array}{ll}
    	r^\beta-(r-\theta)^\beta&\ \mbox{if}\ \theta\leq r\\
    	\infty&\ \mbox{if}\ \theta>r
    \end{array}\right.$$
    and $\kappa_{S_2}(\theta):=h\theta$; thus
    $$\kappa_S(\theta):=\kappa_{S_1}(\theta)+\kappa_{S_2}(-\theta)=\left\{\begin{array}{ll}
    	r^\beta-(r-\theta)^\beta-h\theta&\ \mbox{if}\ \theta\leq r\\
    	\infty&\ \mbox{if}\ \theta>r.
    \end{array}\right.$$
    It is easy to check that, for this example, the function $\Lambda_{\nu,S}$ is essentially smooth if and only if
    $$\lim_{\theta\uparrow r}\Lambda_{\nu,S}^\prime(\theta)=+\infty;$$
    moreover this condition occurs if and only if $\kappa_S(r)>0$, i.e. if and only if $h<r^{\beta-1}$. 
    To better explain this see Figure \ref{fig:essential-smoothness}.
\end{remark}

\begin{figure}[!ht]
	\centering
	\includegraphics[scale=0.47]{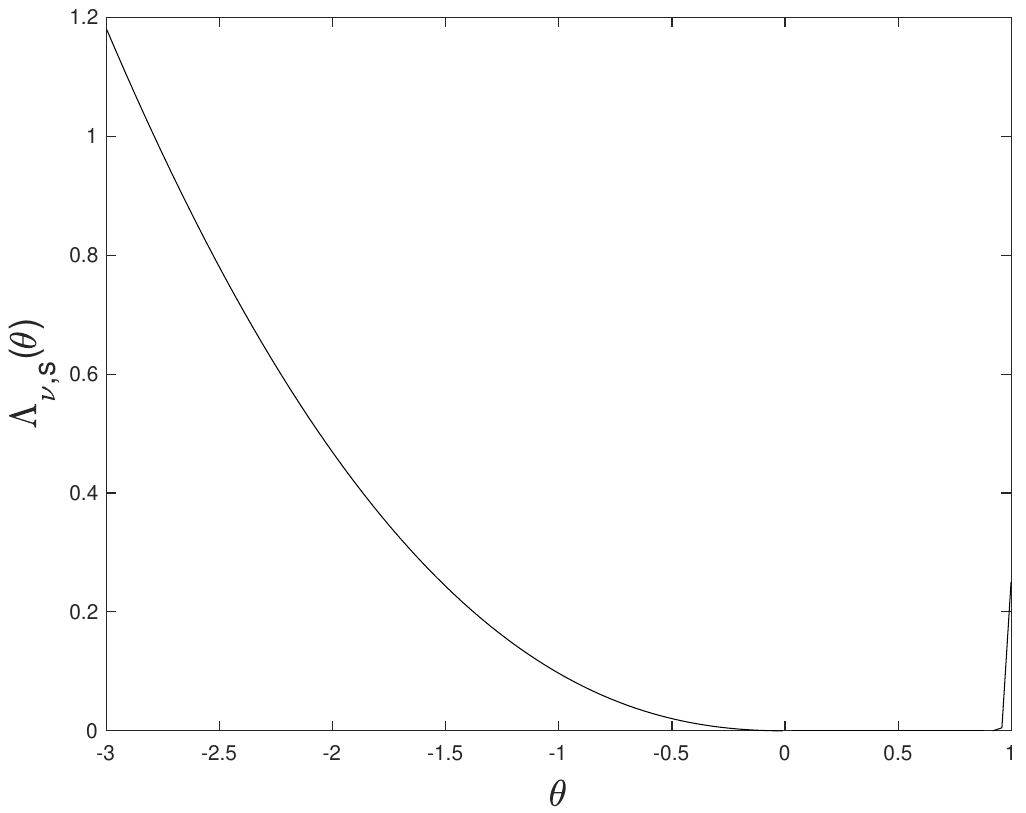}
	\includegraphics[scale=0.47]{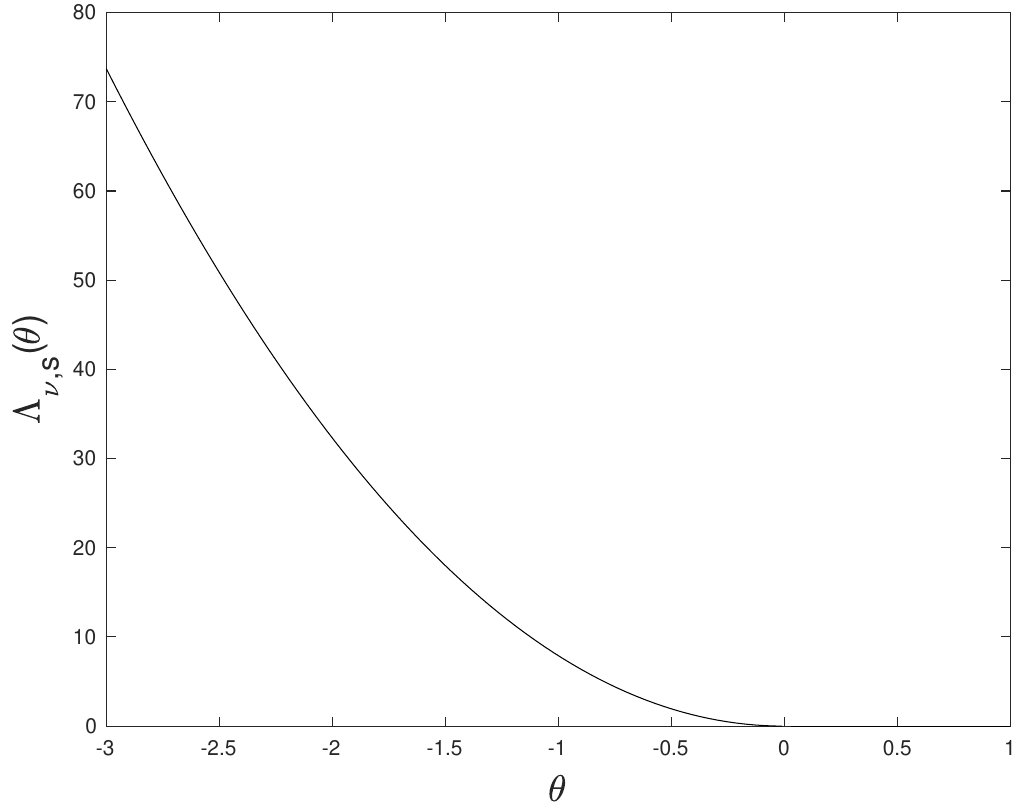}
	\caption{The function $\Lambda_{\nu,S}$ in Remark \ref{rem:es-not-guaranteed} for $\theta\leq r=1$. Numerical values: 
		$\nu=0.5$, $\beta=0.25$; $h=0.5$ on the left, and $h=3$ on the right.}
	\label{fig:essential-smoothness}
\end{figure}

Now we present weak convergence results. In view of these results it is useful to consider the following notation:
\begin{equation}\label{eq:exponents}
	\alpha(\nu):=\left\{\begin{array}{ll}
		1-\nu/2&\ \mbox{if}\ m=0\\
		1-\nu&\ \mbox{if}\ m\neq 0.
	\end{array}\right.
\end{equation}

\begin{proposition}\label{prop:weak-convergence}
Assume that Condition \ref{cond:gen-BM} holds and let $\alpha(\nu)$ be defined in \eqref{eq:exponents}.
We have the following statements.
\begin{itemize}
\item If $m=0$, then $\{t^{\alpha(\nu)}\frac{S(L_\nu(t))}{t}:t>0\}$ converges weakly to
$\sqrt{qL_\nu(1)}Z$, where $Z$ is a standard Normal distributed random variable, and
independent to $L_\nu(1)$.
\item If $m\neq 0$, then $\{t^{\alpha(\nu)}\frac{S(L_\nu(t))}{t}:t>0\}$ converges weakly to
$mL_\nu(1)$.
\end{itemize}
\end{proposition}
\begin{proof}
In both cases $m=0$ and $m\neq 0$ we study suitable limits (as $t\to\infty$) in terms of the moment generating function in 
\eqref{eq:MGF-S}; so, when we take the limit, we have to take into account \eqref{eq:ML-asymptotics}.

If $m=0$, then we have
\begin{multline*}
\mathbb{E}\left[e^{\theta t^{\alpha(\nu)}\frac{S(L_\nu(t))}{t}}\right]
=\mathbb{E}\left[e^{\theta\frac{S(L_\nu(t))}{t^{\nu/2}}}\right]=E_\nu\left(\kappa_S\left(\frac{\theta}{t^{\nu/2}}\right)t^\nu\right)\\ 
=E_\nu\left(\left(\frac{q\theta^2}{2t^\nu}+o\left(\frac{1}{t^\nu}\right)\right)t^\nu\right)
\to E_\nu\left(\frac{q\theta^2}{2}\right)\ \mbox{for all}\ \theta\in\mathbb{R}.
\end{multline*}
Thus the desired weak convergence is proved noting that (here we take into account 
\eqref{eq:MGF-inverse-stable-sub})
$$\mathbb{E}\left[e^{\theta\sqrt{qL_\nu(1)}Z}\right]
=\mathbb{E}\left[e^{\frac{\theta^2q}{2}L_\nu(1)}\right]
=E_\nu\left(\frac{q\theta^2}{2}\right)\ \mbox{for all}\ \theta\in\mathbb{R}.$$

If $m\neq 0$, then we have
\begin{multline*}
\mathbb{E}\left[e^{\theta t^{\alpha(\nu)}\frac{S(L_\nu(t))}{t}}\right]
=\mathbb{E}\left[e^{\theta\frac{S(L_\nu(t))}{t^\nu}}\right]
=E_\nu\left(\kappa_S\left(\frac{\theta}{t^\nu}\right)t^\nu\right)\\
=E_\nu\left(\left(\frac{m\theta}{t^\nu}+o\left(\frac{1}{t^\nu}\right)\right)t^\nu\right)
\to E_\nu\left(m\theta\right)\ \mbox{for all}\ \theta\in\mathbb{R}.
\end{multline*}
Thus the desired weak convergence is proved by \eqref{eq:MGF-inverse-stable-sub}.
\end{proof}

Now we present the non-central moderate deviation results.

\begin{proposition}\label{prop:ncMD}
Assume that Condition \ref{cond:gen-BM} holds and let $\alpha(\nu)$ be defined in \eqref{eq:exponents}.
Moreover assume that $q>0$ if $m=0$. Then, for every family of positive numbers $\{a_t:t>0\}$ such that \eqref{eq:MDconditions} 
holds, the family of random variables $\left\{\frac{(a_tt)^{\alpha(\nu)}S(L_\nu(t))}{t}:t>0\right\}$ 
satisfies the LDP with speed $1/a_t$ and good rate function $I_{\mathrm{MD}}(\cdot;m)$ defined by:
$$\left.\begin{array}{cc}
\mbox{if}\ m=0,&\ I_{\mathrm{MD}}(x;0):=((\nu/2)^{\nu/(2-\nu)}-(\nu/2)^{2/(2-\nu)})\left(\frac{2x^2}{q}\right)^{1/(2-\nu)};\\
\mbox{if}\ m\neq 0,&\ I_{\mathrm{MD}}(x;m):=\left\{\begin{array}{ll}
(\nu^{\nu/(1-\nu)}-\nu^{1/(1-\nu)})\left(\frac{x}{m}\right)^{1/(1-\nu)}&\ \mbox{if}\ \frac{x}{m}\geq 0\\
\infty&\ \mbox{if}\ \frac{x}{m}<0.
\end{array}\right.
\end{array}\right.$$
\end{proposition}
\begin{proof}
For every $m\in\mathbb{R}$ we apply the G\"artner Ellis Theorem (Theorem \ref{th:GE}). So we have to show that
we can consider the function $\Lambda_{\nu,m}$ defined by
$$\Lambda_{\nu,m}(\theta):=\lim_{t\to\infty}
\frac{1}{1/a_t}\log\mathbb{E}\left[e^{\frac{\theta}{a_t}\frac{(a_tt)^{\alpha(\nu)}S(L_\nu(t))}{t}}\right]
\ \mbox{for all}\ \theta\in\mathbb{R},$$
or equivalently
$$\Lambda_{\nu,m}(\theta)
:=\lim_{t\to\infty}a_t\log E_\nu\left(\kappa_S\left(\frac{\theta}{(a_tt)^{1-\alpha(\nu)}}\right)t^\nu\right)
\ \mbox{for all}\ \theta\in\mathbb{R};$$
in particular we refer to \eqref{eq:ML-asymptotics} when we take the limit. Moreover, again for every 
$m\in\mathbb{R}$, we shall see that the function $\Lambda_{\nu,m}$ satisfies the hypotheses of the G\"artner 
Ellis Theorem (this can be checked by considering the expressions of the function $\Lambda_{\nu,m}$ below), and
therefore the LDP holds with good rate function $I_{\mathrm{MD}}(\cdot;m)$ defined by
\begin{equation}\label{eq:MD-Legendre-transform}
I_{\mathrm{MD}}(x;m):=\sup_{\theta\in\mathbb{R}}\{\theta x-\Lambda_{\nu,m}(\theta)\}.
\end{equation}
Then, as we shall explain below, for every $m\in\mathbb{R}$ the rate function expression in \eqref{eq:MD-Legendre-transform}
coincides with the rate function $I_{\mathrm{MD}}(\cdot;m)$ in the statement.

If $m=0$ we have
\begin{multline*}
a_t\log E_\nu\left(\kappa_S\left(\frac{\theta}{(a_tt)^{1-\alpha(\nu)}}\right)t^\nu\right)\\
=a_t\log E_\nu\left(\left(\frac{q\theta^2}{2(a_tt)^\nu}+o\left(\frac{1}{(a_tt)^\nu}\right)\right)t^\nu\right)
=a_t\log E_\nu\left(\frac{1}{a_t^\nu}\left(\frac{q\theta^2}{2}+(a_tt)^\nu o\left(\frac{1}{(a_tt)^\nu}\right)\right)\right),
\end{multline*}
and therefore
$$\lim_{t\to\infty}a_t\log E_\nu\left(\kappa_S\left(\frac{\theta}{(a_tt)^{1-\alpha(\nu)}}\right)t^\nu\right)
=\left(\frac{q\theta^2}{2}\right)^{1/\nu}=:\Lambda_{\nu,0}(\theta)\ \mbox{for all}\ \theta\in\mathbb{R};$$
thus the desired LDP holds with good rate function $I_{\mathrm{MD}}(\cdot;0)$ defined by \eqref{eq:MD-Legendre-transform}
which coincides with the rate function expression in the statement (indeed one can check that, for all $x\in\mathbb{R}$, 
the supremum in \eqref{eq:MD-Legendre-transform} is attained at 
$\theta=\theta_x:=\left(\frac{2}{q}\right)^{1/(2-\nu)}\left(\frac{\nu x}{2}\right)^{\nu/(2-\nu)}$).

If $m\neq 0$ we have
\begin{multline*}
a_t\log E_\nu\left(\kappa_S\left(\frac{\theta}{(a_tt)^{1-\alpha(\nu)}}\right)t^\nu\right)\\
=a_t\log E_\nu\left(\left(\frac{\theta m}{(a_tt)^\nu}+o\left(\frac{1}{(a_tt)^\nu}\right)\right)t^\nu\right)
=a_t\log E_\nu\left(\frac{1}{a_t^\nu}\left(\theta m+(a_tt)^\nu o\left(\frac{1}{(a_tt)^\nu}\right)\right)\right)
\end{multline*}
and therefore
$$\lim_{t\to\infty}a_t\log E_\nu\left(\kappa_S\left(\frac{\theta}{(a_tt)^{1-\alpha(\nu)}}\right)t^\nu\right)
=\left\{\begin{array}{ll}
(\theta m)^{1/\nu}&\ \mbox{if}\ \theta m\geq 0\\
0&\ \mbox{if}\ \theta m<0
\end{array}\right.=:\Lambda_{\nu,m}(\theta)\ \mbox{for all}\ \theta\in\mathbb{R};$$
thus the desired LDP holds with good rate function $I_{\mathrm{MD}}(\cdot;m)$ defined by \eqref{eq:MD-Legendre-transform}
which coincides with the rate function expression in the statement (indeed one can check that the supremum in 
\eqref{eq:MD-Legendre-transform} is attained at 
$\theta=\theta_x:=\frac{1}{m}\left(\frac{\nu x}{m}\right)^{\nu/(1-\nu)}$ for $\frac{x}{m}\geq 0$, and it is equal to
infinity for $\frac{x}{m}<0$ by letting $\theta\to\infty$ if $m<0$, and by letting $\theta\to-\infty$ if $m>0$).
\end{proof}

\begin{remark}\label{rem:*}
	As we said above, the results in this section provide a generalization of the results in \cite{BeghinMacciSPL2022} 
	in which $\{S(t):t\geq 0\}$ is a compound Poisson process. More precisely we mean that $S(t):=\sum_{k=1}^{N(t)}X_k$,
	where $\{X_n:n\geq 1\}$ are i.i.d. real valued light tailed random variables with finite mean $\mu$ and finite variance
	$\sigma^2$ (in \cite{BeghinMacciSPL2022} it was requested that $\sigma^2>0$ to avoid trivialities), independent of a 
	Poisson process $\{N(t):t\geq 0\}$ with intensity $\lambda>0$. Therefore 
	$\kappa_S(\theta)=\lambda(\mathbb{E}[e^{\theta X_1}]-1)$ for all $\theta\in\mathbb{R}$; moreover
	(see $m$ and $q$ in Condition \ref{cond:real-lighttailed-Levyprocess}) $m=\lambda\mu$ and $q=\lambda(\sigma^2+\mu^2)$.
	
	Moreover we can adapt the content of Remark 3.4 in \cite{BeghinMacciSPL2022} and we can say that, for every $m\in\mathbb{R}$ 
	(thus the case $m=0$ can be also considered), we have $I_{\mathrm{MD}}(x;m)=I_{\mathrm{MD}}(-x;-m)$ for every $x\in\mathbb{R}$.
	Finally, if we refer to $\lambda$ and $\mu$ at the beginning of this remark, we recover the rate functions in Proposition 3.3
	in \cite{BeghinMacciSPL2022} as follows:
	\begin{itemize}
		\item if $m=\lambda\mu=0$ (and therefore $\mu=0$ and $q=\lambda\sigma^2$), then $I_{\mathrm{MD}}(\cdot;0)$ in Proposition
		\ref{prop:ncMD} in this paper coincides with $I_{\mathrm{MD},0}$ in Proposition 3.3 in \cite{BeghinMacciSPL2022};
		\item if $m=\lambda\mu\neq 0$ (and therefore $\mu\neq 0$), then $I_{\mathrm{MD}}(\cdot;m)$ in Proposition
		\ref{prop:ncMD} in this paper coincides with $I_{\mathrm{MD},\mu}$ in Proposition 3.3 in \cite{BeghinMacciSPL2022}.
	\end{itemize}
\end{remark}

\begin{remark}\label{rem:q-different-from-zero}
	In Proposition \ref{prop:ncMD} we have assumed that $q\neq 0$ when $m=0$. Indeed, if $q=0$ and $m=0$, the process 
	$\{S(L_\nu(t)):t\geq 0\}$ in Propositions \ref{prop:LD}, \ref{prop:weak-convergence} and \ref{prop:ncMD} is identically equal
	to zero (because $S(t)=0$ for all $t\geq 0$) and the weak convergence in Proposition \ref{prop:weak-convergence} (for $m=0$) 
	is towards a constant random variable (i.e. the costant random variable equal to zero). Moreover, again if $q=0$ and $m=0$,
	the rate function $I_{\mathrm{MD}}(x;0)$ in Proposition \ref{prop:ncMD} is not well-defined (because there is a denominator 
	equal to zero).
\end{remark}

\section{Results with two independent random time-changes}\label{sec:gen-results-Lee-Macci}
Throughout this section we assume that the following condition holds.

\begin{condition}\label{cond:gen-LM}
	Let $\{S(t):t\geq 0\}$ be a real-valued Lévy process as in Condition \ref{cond:real-lighttailed-Levyprocess},
    and let $\{L_{\nu_1}^{(1)}(t):t\geq 0\}$ and $\{L_{\nu_2}^{(2)}(t):t\geq 0\}$ be two independent inverses of stable
    subordinators for $\nu_1,\nu_2\in(0,1)$, and independent of $\{S(t):t\geq 0\}$.	We assume that $\{S(t):t\geq 0\}$ has bounded
    variation, and it is not a subordinator.
\end{condition}

We have the following consequence of Condition \ref{cond:gen-LM}.

\begin{lemma}\label{lem:Chi}
	Assume that Condition \ref{cond:gen-LM} holds. Then there exists two non-null independent subordinators $\{S_1(t):t\geq 0\}$ 
	and $\{S_2(t):t\geq 0\}$ such that $\{S(t):t\geq 0\}$ is distributed as $\{S_1(t)-S_2(t):t\geq 0\}$.
\end{lemma}
We can assume that the statement in Lemma \ref{lem:Chi} is known even if we do not have an exact reference for that result
(however a statement of this kind appears in the Introduction of \cite{Chi}). The idea of the proof is the following. If $\Pi(dx)$
is the Lévy measure of a Lévy process with bounded variation, then $1_{(0,\infty)}(x)\Pi(dx)$ and $1_{(-\infty,0)}(x)\Pi(dx)$ are 
again Lévy measures of Lévy processes with bounded variation; thus $1_{(0,\infty)}(x)\Pi(dx)$ is the Lévy measure associated 
to the subordinator $\{S_1(t):t\geq 0\}$, $1_{(-\infty,0)}(x)\Pi(dx)$ is the Lévy measure associated to the opposite of the 
subordinator $\{S_2(t):t\geq 0\}$, and $\{S_1(t):t\geq 0\}$ and $\{S_2(t):t\geq 0\}$ are independent.

\begin{remark}\label{rem:cond:gen-LM-consequences}
	Let $\kappa_{S_1}$ and $\kappa_{S_2}$ be the analogue of the function $\kappa_S$ for the process $\{S(t):t\geq 0\}$ in 
	Condition \ref{cond:real-lighttailed-Levyprocess}, i.e. the functions defined by
	$$\kappa_{S_i}(\theta):=\log\mathbb{E}[e^{\theta S_i(1)}]\quad(\mbox{for}\ i=1,2),$$
	where $\{S_1(t):t\geq 0\}$ and $\{S_2(t):t\geq 0\}$ are the subordinators in Lemma \ref{lem:Chi}. In particular both functions are 
	finite in a neighborhood of the origin. Then, if we set
	$$m_i=\kappa_{S_i}^\prime(0)\quad\mbox{and}\quad q_i=\kappa_{S_i}^{\prime\prime}(0)\quad (\mbox{for}\ i\in\{1,2\}),$$
	we have (we recall that $\kappa_S(\theta)=\kappa_{S_1}(\theta)+\kappa_{S_2}(-\theta)$ for all $\theta\in\mathbb{R}$)
	$$m=\kappa_S^\prime(0)=m_1-m_2\quad\mbox{and}\quad q=\kappa_S^{\prime\prime}(0)=q_1+q_2.$$
	We recall that, since $\{S_1(t):t\geq 0\}$ and $\{S_2(t):t\geq 0\}$ are non-trivial subordinators, then $m_1,m_2>0$.
\end{remark}

The next Propositions \ref{prop:LD-bis}, \ref{prop:weak-convergence-bis} and \ref{prop:ncMD-bis} provide a generalization
of Propositions 3.1, 3.2 and 3.3 in \cite{LeeMacci}, respectively, in which $\{S(t):t\geq 0\}$ is a Skellam process
(and therefore $\{S_1(t):t\geq 0\}$ and $\{S_2(t):t\geq 0\}$ are two Poisson processes with intensities $\lambda_1$ and 
$\lambda_2$, respectively). We start with the reference LDP for the convergence in probability to zero of $\left\{\frac{S_1(L_{\nu_1}^{(1)}(t))-S_2(L_{\nu_2}^{(2)}(t))}{t}:t>0\right\}$. In this first result the case
$\nu_1\neq\nu_2$ is allowed.

\begin{proposition}\label{prop:LD-bis}
	Assume that Condition \ref{cond:gen-LM} holds (therefore we can refer to the independent subordinators
	$\{S_1(t):t\geq 0\}$ and $\{S_2(t):t\geq 0\}$ in Lemma \ref{lem:Chi}). Let $\Psi_{\nu_1,\nu_2}$ be the function defined by
	\begin{equation}\label{eq:LD-GE-limit-type1}
		\Psi_{\nu_1,\nu_2}(\theta):=\left\{\begin{array}{ll}
			(\kappa_{S_1}(\theta))^{1/\nu_1}&\ \mbox{if}\ \theta\geq 0\\
			(\kappa_{S_2}(-\theta))^{1/\nu_2}&\ \mbox{if}\ \theta<0.
		\end{array}\right.
	\end{equation}
    Then $\left\{\frac{S_1(L_{\nu_1}^{(1)}(t))-S_2(L_{\nu_2}^{(2)}(t))}{t}:t>0\right\}$ satisfies the LDP with 
	speed $v_t=t$ and good rate function $J_{\mathrm{LD}}$ defined by
	\begin{equation}\label{eq:LD-type1-rf-Legendre-transform}
		J_{\mathrm{LD}}(x):=\sup_{\theta\in\mathbb{R}}\{\theta x-\Psi_{\nu_1,\nu_2}(\theta)\}.
	\end{equation}
\end{proposition}
\begin{proof}
	We prove this proposition by applying the G\"artner Ellis Theorem. More precisely we have to show that
	\begin{equation}\label{eq:LD-GET-limit-type1}
		\lim_{t\to\infty}\frac{1}{t}\log\mathbb{E}\left[e^{t\theta\frac{S_1(L_{\nu_1}^{(1)}(t))-S_2(L_{\nu_2}^{(2)}(t))}{t}}\right]
		=\Psi_{\nu_1,\nu_2}(\theta)\ (\mbox{for all}\ \theta\in\mathbb{R}),
	\end{equation}
	where $\Psi_{\nu_1,\nu_2}$ is the function in \eqref{eq:LD-GE-limit-type1}.
	
	The case $\theta=0$ is immediate. For $\theta\neq 0$ we have
	$$\log\mathbb{E}\left[e^{t\theta\frac{S_1(L_{\nu_1}^{(1)}(t))-S_2(L_{\nu_2}^{(2)}(t))}{t}}\right]
	=\log E_{\nu_1}(\kappa_{S_1}(\theta)t^{\nu_1})+\log E_{\nu_2}(\kappa_{S_2}(-\theta)t^{\nu_2}).$$
	Then, by taking into account the asymptotic behaviour of the Mittag-Leffler function in \eqref{eq:ML-asymptotics}, we have
	$$\lim_{t\to\infty}\frac{1}{t}\log E_{\nu_1}(\kappa_{S_1}(\theta)t^{\nu_1})+\lim_{t\to\infty}
	\frac{1}{t}\log E_{\nu_2}(\kappa_{S_2}(-\theta)t^{\nu_2})=(\kappa_{S_1}(\theta))^{1/\nu_1}\ \mbox{for}\ \theta>0,$$
	and
	$$\lim_{t\to\infty}\frac{1}{t}\log E_{\nu_1}(\kappa_{S_1}(\theta)t^{\nu_1})+\lim_{t\to\infty}
	\frac{1}{t}\log E_{\nu_2}(\kappa_{S_2}(-\theta)t^{\nu_2})=(\kappa_{S_2}(-\theta))^{1/\nu_2}\ \mbox{for}\ \theta<0;$$
	thus the limit in \eqref{eq:LD-GET-limit-type1} is checked. Finally the desired LDP holds because the function 
	$\Psi_{\nu_1,\nu_2}$ is essentially smooth. The essential smoothness of $\Psi_{\nu_1,\nu_2}$ trivially holds if 
	$\Psi_{\nu_1,\nu_2}(\theta)$ is finite everywhere (and differentiable). So now we assume that $\Psi_{\nu_1,\nu_2}(\theta)$ is not
	finite everywhere. For $i=1,2$ we have
	$$\frac{d}{d\theta}(\kappa_{S_i}(\theta))^{1/\nu_i}=\frac{1}{\nu_i}(\kappa_{S_i}(\theta))^{1/\nu_i-1}\kappa_{S_i}^\prime(\theta),$$
    and therefore the range of values of each one of these derivatives (for $\theta\geq 0$ such that $\kappa_{S_i}(\theta)<\infty$)
    is $[0,\infty)$; therefore the range of values of $\Psi_{\nu_1,\nu_2}^\prime(\theta)$ (for $\theta\in\mathbb{R}$ such that $\Psi_{\nu_1,\nu_2}(\theta)<\infty$) is $(-\infty,\infty)$, and the essential smoothness of $\Psi_{\nu_1,\nu_2}$ is proved.
\end{proof}

From now on we assume that $\nu_1$ and $\nu_2$ coincide, and therefore we simply consider the symbol $\nu$, where
$\nu=\nu_1=\nu_2$. Moreover we set
\begin{equation}\label{eq:exponent}
	\alpha_1(\nu):=1-\nu.
\end{equation}

\begin{proposition}\label{prop:weak-convergence-bis}
	Assume that Condition \ref{cond:gen-LM} holds (therefore we can refer to the independent subordinators
	$\{S_1(t):t\geq 0\}$ and $\{S_2(t):t\geq 0\}$ in Lemma \ref{lem:Chi}). Moreover assume that $\nu_1=\nu_2=\nu$ for some 
    $\nu\in(0,1)$ and let $\alpha_1(\nu)$ be defined in \eqref{eq:exponent}. Then $\{t^{\alpha_1(\nu)}\frac{S_1(L_\nu^{(1)}(t))-S_2(L_\nu^{(2)}(t))}{t}:t>0\}$
	converges weakly to $m_1L_\nu^{(1)}(1)-m_2L_\nu^{(2)}(1)$.
\end{proposition}
\begin{proof}
	We have to check that
	$$\lim_{t\to\infty}\mathbb{E}\left[e^{\theta t^{\alpha_1(\nu)}\frac{S_1(L_\nu^{(1)}(t))-S_2(L_\nu^{(2)}(t))}{t}}\right]=
	\underbrace{\mathbb{E}\left[e^{\theta (m_1L_\nu^{(1)}(1)-m_2L_\nu^{(2)}(1))}\right]}
	_{=E_\nu(m_1\theta)E_\nu(-m_2\theta)}\ (\mbox{for all}\ \theta\in\mathbb{R})$$
	(here we take into account that $L_\nu^{(1)}(1)$ and $L_\nu^{(2)}(1)$ are i.i.d., and the expression of the
	moment generating function in \eqref{eq:MGF-inverse-stable-sub}). This can be readily done noting that
	\begin{multline*}
		\mathbb{E}\left[e^{\theta t^{\alpha_1(\nu)}\frac{S_1(L_\nu^{(1)}(t))-S_2(L_\nu^{(2)}(t))}{t}}\right]
		=\mathbb{E}\left[e^{\theta\frac{S_1(L_\nu^{(1)}(t))-S_2(L_\nu^{(2)}(t))}{t^\nu}}\right]\\
		=E_\nu\left(\kappa_{S_1}\left(\frac{\theta}{t^\nu}\right)t^\nu\right)
		E_\nu\left(\kappa_{S_2}\left(-\frac{\theta}{t^\nu}\right)t^\nu\right)\\
		=E_\nu\left(\left(m_1\frac{\theta}{t^\nu}+o\left(\frac{1}{t^\nu}\right)\right)t^\nu\right)
		E_\nu\left(\left(-m_2\frac{\theta}{t^\nu}+o\left(\frac{1}{t^\nu}\right)t^\nu\right)\right),
	\end{multline*}
	and we get the desired limit letting $t$ go to infinity (for each fixed $\theta\in\mathbb{R}$).
\end{proof}

\begin{proposition}\label{prop:ncMD-bis}
	Assume that Condition \ref{cond:gen-LM} holds (therefore we can refer to the independent subordinators
		$\{S_1(t):t\geq 0\}$ and $\{S_2(t):t\geq 0\}$ in Lemma \ref{lem:Chi}). Moreover assume that 
	$\nu_1=\nu_2=\nu$ for some $\nu\in(0,1)$ and let $\alpha_1(\nu)$ be defined in \eqref{eq:exponent}. Then, 
	for every family of positive numbers $\{a_t:t>0\}$ such that \eqref{eq:MDconditions} holds, the family of 
	random variables $\left\{(a_tt)^{\alpha_1(\nu)}\frac{S_1(L_\nu^{(1)}(t))-S_2(L_\nu^{(2)}(t))}{t}:t>0\right\}$ satisfies
	the LDP with speed $1/a_t$ and good rate function $J_{\mathrm{MD}}$ defined by
	$$J_{\mathrm{MD}}(x):=\left\{\begin{array}{ll}
		(\nu^{\nu/(1-\nu)}-\nu^{1/(1-\nu)})\left(\frac{x}{m_1}\right)^{1/(1-\nu)}&\ \mbox{if}\ x\geq 0\\
		(\nu^{\nu/(1-\nu)}-\nu^{1/(1-\nu)})\left(-\frac{x}{m_2}\right)^{1/(1-\nu)}&\ \mbox{if}\ x<0.
	\end{array}\right.$$
\end{proposition}
\begin{proof}
	We prove this proposition by applying the G\"artner Ellis Theorem. More precisely we have to show that
	\begin{equation}\label{eq:NCMD-GET-limit-type1}
		\lim_{t\to\infty}\frac{1}{1/a_t}\log
		\mathbb{E}\left[e^{\frac{\theta}{a_t}(a_tt)^{\alpha_1(\nu)}\frac{S_1(L_\nu^{(1)}(t))-S_2(L_\nu^{(2)}(t))}{t}}\right]
		=\widetilde{\Psi}_\nu(\theta)\ (\mbox{for all}\ \theta\in\mathbb{R}),
	\end{equation}
	where $\widetilde{\Psi}_\nu$ is the function defined by
	$$\widetilde{\Psi}_\nu(\theta):=\left\{\begin{array}{ll}
		(m_1\theta)^{1/\nu}&\ \mbox{if}\ \theta\geq 0\\
		(-m_2\theta)^{1/\nu}&\ \mbox{if}\ \theta<0;
	\end{array}\right.$$
	indeed, since the function $\widetilde{\Psi}_\nu$ is finite (for all $\theta\in\mathbb{R}$) and differentiable, 
	the desired LDP holds noting that the Legendre-Fenchel transform $\widetilde{\Psi}_\nu^*$ of 
	$\widetilde{\Psi}_\nu$, i.e. the function $\widetilde{\Psi}_\nu^*$ defined by
	\begin{equation}\label{eq:MD-Legendre-transform-bis}
		\widetilde{\Psi}_\nu^*(x):=
		\sup_{\theta\in\mathbb{R}}\{\theta x-\widetilde{\Psi}_\nu(\theta)\}\ (\mbox{for all}\ x\in\mathbb{R}),
	\end{equation}
	coincides with the function $J_{\mathrm{MD}}$ in the statement of the proposition (for $x=0$ the
	supremum in \eqref{eq:MD-Legendre-transform-bis} is attained at $\theta=0$, for $x>0$ that supremum is attained at 
	$\theta=\frac{1}{m_1}(\frac{\nu x}{m_1})^{\nu/(1-\nu)}$, for $x<0$ that supremum is attained at 
	$\theta=-\frac{1}{m_2}(-\frac{\nu x}{m_2})^{\nu/(1-\nu)}$).
	
	So we conclude the proof by checking the limit in \eqref{eq:NCMD-GET-limit-type1}. The case $\theta=0$ 
	is immediate. For $\theta\neq 0$ we have
	\begin{multline*}
		\log\mathbb{E}\left[e^{\frac{\theta}{a_t}(a_tt)^{\alpha_1(\nu)}\frac{S_1(L_\nu^{(1)}(t))-S_2(L_\nu^{(2)}(t))}{t}}\right]
		=\log\mathbb{E}\left[e^{\theta\frac{S_1(L_\nu^{(1)}(t))-S_2(L_\nu^{(2)}(t))}{(a_tt)^\nu}}\right]\\
		=\log E_\nu\left(\kappa_{S_1}\left(\frac{\theta}{(a_tt)^\nu}\right)t^\nu\right)+
		\log E_\nu\left(\kappa_{S_2}\left(-\frac{\theta}{(a_tt)^\nu}\right)t^\nu\right)\\
		=\log E_\nu\left(\left(m_1\frac{\theta}{(a_tt)^\nu}+o\left(\frac{1}{(a_tt)^\nu}\right)\right)t^\nu\right)
		+\log E_\nu\left(\left(-m_2\frac{\theta}{(a_tt)^\nu}+o\left(\frac{1}{(a_tt)^\nu}\right)\right)t^\nu\right)\\
		=\log E_\nu\left(\frac{m_1}{a_t^\nu}\left(\theta+(a_tt)^\nu o\left(\frac{1}{(a_tt)^\nu}\right)\right)\right)
		+\log E_\nu\left(\frac{m_2}{a_t^\nu}\left(-\theta+(a_tt)^\nu o\left(\frac{1}{(a_tt)^\nu}\right)\right)\right).
	\end{multline*}
	Then, by taking into account the asymptotic behaviour of the Mittag-Leffler function in \eqref{eq:ML-asymptotics},
	we have
	$$\lim_{t\to\infty}\frac{1}{1/a_t}\log \mathbb{E}\left[e^{\frac{\theta}{a_t}(a_tt)^{\alpha_1(\nu)}\frac{S_1(L_\nu^{(1)}(t))-S_2(L_\nu^{(2)}(t))}{t}}\right]
	=(m_1\theta)^{1/\nu}\quad \mbox{for}\ \theta>0,$$
	and
	$$\lim_{t\to\infty}\frac{1}{1/a_t}\log \mathbb{E}\left[e^{\frac{\theta}{a_t}(a_tt)^{\alpha_1(\nu)}\frac{S_1(L_\nu^{(1)}(t))-S_2(L_\nu^{(2)}(t))}{t}}\right]
	=(-m_2\theta)^{1/\nu}\quad \mbox{for}\ \theta<0.$$
	Thus the limit in \eqref{eq:NCMD-GET-limit-type1} is checked.
\end{proof}

\section{A discussion on possible generalizations}\label{sec:referee}
In Sections \ref{sec:gen-results-Beghin-Macci-2022} and \ref{sec:gen-results-Lee-Macci} we have proved two moderate deviation
results, i.e. two collections of three results: a reference LDP (Propositions \ref{prop:LD} and \ref{prop:LD-bis}),
a weak convergence result (Propositions \ref{prop:weak-convergence} and \ref{prop:weak-convergence-bis}), and a collection
of LDPs which depend on some positive scalings $\{a_t:t>0\}$ which satisfies condition \eqref{eq:MDconditions} (Propositions
\ref{prop:ncMD} and \ref{prop:ncMD-bis}).

Then one can wonder if it is possible to consider more general time-changes to have a generalization of at least one of these
three results in each collection. This question is quite natural because, if one looks at the results presented above, the
asymptotic behaviour of Mittag-Leffler function (see in \eqref{eq:ML-asymptotics}) seems to plays a role in the applications 
of G\"artner Ellis Theorem (thus in all the results, except those of weak convergence).

In what follows we consider random time-changes which satisfy the following condition.

\begin{condition}\label{cond:gen-iss}
	Let $\{L_f(t):t\geq 0\}$ be a nonnegative process such that there exists
	$$\lim_{t\to\infty}\frac{1}{t}\log\mathbb{E}[e^{\rho L_f(t)}]=\left\{\begin{array}{ll}
		f(\rho)&\ \mbox{if}\ \rho\geq 0\\
		0&\ \mbox{if}\ \rho<0
	\end{array}\right.=:\Upsilon_f(\rho),$$
	where $f$ is a regular, convex and non-decreasing (real-valued) function defined on $[0,\infty)$ such that 
	$f(0)=0$ and $f^\prime(0)=0$ (here $f^\prime(0)$ is the right derivative of $f$ at $\eta=0$).
\end{condition}

This condition is quite natural way if one deals with inverse of heavy-tailed subordinators, indeed it is 
satisfied by inverse stable subordinators. This will be explained in the following remark.

\begin{remark}
	Condition \ref{cond:gen-iss} holds if $\{L_f(t):t\geq 0\}$ is the inverse of subordinator $\{V(t):t\geq 0\}$
	such that $\mathbb{E}[V(1)]=\infty$. In such a case, if we consider the function
	$$\kappa_V(\xi):=\log\mathbb{E}[e^{\xi V(1)}],$$
	we have a regular and increasing function for $\xi\leq 0$, $\kappa_V(0)=0$ and $\kappa_V^\prime(0-)=\infty$ 
	(here $\kappa_V^\prime(0-)=\infty$ is the left derivative of $\kappa_V$ at $\xi=0$) and
	$\kappa_V(\xi)=\infty$ for $\xi\geq 0$. Then the restriction of $\kappa_V$ over $(-\infty,0]$ is invertible,
	and assume values in $(-\infty,0]$; so we denote such inverse function by $\kappa_V^{-1}$. Then one can check that
	$$f(\rho)=-\kappa_V^{-1}(-\rho)\ (\mbox{for}\ \rho\geq 0),$$
	and this agrees with formulas (12)-(13) in \cite{GlynnWhitt}. In particular we recover the case of stable 
	subordinators with with $\kappa_V(\xi)=-(-\xi)^\nu$ for $\xi\leq 0$ and $f(\rho)=\rho^{1/\nu}$ for $\rho\geq 0$.
\end{remark}

Then we can present a generalization of the reference LDPs presented in Propositions \ref{prop:LD} and 
\ref{prop:LD-bis}. The proofs are very similar to the ones presented for those propositions, and we omit the details.

\begin{proposition}\label{prop:LD-gen}
	Let $\{S(t):t\geq 0\}$ be a real-valued Lévy process as in Condition \ref{cond:real-lighttailed-Levyprocess};
    moreover let $\{L_f(t):t\geq 0\}$ be real-valued Lévy process as in Condition \ref{cond:gen-iss}, and 
    assume that $\{S(t):t\geq 0\}$ and $\{L_f(t):t\geq 0\}$ are independent. Furthermore we consider the function
		$$\Lambda_{f,S}(\theta):=\Upsilon_f(\kappa_S(\theta))=\left\{\begin{array}{ll}
			f(\kappa_S(\theta))&\ \mbox{if}\ \kappa_S(\theta)\geq 0\\
			0&\ \mbox{if}\ \kappa_S(\theta)<0,
		\end{array}\right.$$
	and assume that it is an essentially smooth function. Then $\left\{\frac{S(L_f(t))}{t}:t>0\right\}$ 
	satisfies the LDP with speed $v_t=t$ and good rate function $I_{\mathrm{LD},f}$ defined by
		$$I_{\mathrm{LD},f}(x):=\sup_{\theta\in\mathbb{R}}\{\theta x-\Lambda_{f,S}(\theta)\}.$$
\end{proposition}

\begin{proposition}\label{prop:LD-bis-gen}
	Let $\{S(t):t\geq 0\}$ be a real-valued Lévy process as in Condition \ref{cond:real-lighttailed-Levyprocess},
	and let $\{L_{f_1}^{(1)}(t):t\geq 0\}$ and $\{L_{f_2}^{(2)}(t):t\geq 0\}$ be two independent processes as in
	Condition \ref{cond:gen-iss} (for some functions $f_1$ and $f_2$, respectively), and independent of 
	$\{S(t):t\geq 0\}$.	Moreover assume that $\{S(t):t\geq 0\}$ has bounded variation, and it is not a subordinator
	(therefore we can refer to the independent subordinators $\{S_1(t):t\geq 0\}$ and $\{S_2(t):t\geq 0\}$ in 
	Lemma \ref{lem:Chi}). Then let $\Psi_{f_1,f_2}$ be the function defined by
    	$$\Psi_{f_1,f_2}(\theta):=\left\{\begin{array}{ll}
    		f_1(\kappa_{S_1}(\theta))&\ \mbox{if}\ \theta\geq 0\\
    		f_2(\kappa_{S_2}(-\theta))&\ \mbox{if}\ \theta<0,
    	\end{array}\right.$$
    and assume that is essentially smooth. Then $\left\{\frac{S_1(L_{f_1}^{(1)}(t))-S_2(L_{f_2}^{(2)}(t))}{t}:t>0\right\}$ 
    satisfies the LDP with speed $v_t=t$ and good rate function $J_{\mathrm{LD},f_1,f_2}$ defined by
    	$$J_{\mathrm{LD},f_1,f_2}(x):=\sup_{\theta\in\mathbb{R}}\{\theta x-\Psi_{f_1,f_2}(\theta)\}.$$
\end{proposition}

Finally we discuss a possible way to generalize Propositions \ref{prop:ncMD} and \ref{prop:ncMD-bis}
when Condition \ref{cond:gen-iss} holds, together with some possible further conditions. Firstly one should have the analogue of 
Propositions \ref{prop:weak-convergence} and \ref{prop:weak-convergence-bis}; thus one should have the weak convergence of
$$\left\{\frac{S(L_f(t))}{t^{1-\alpha(f)}}:t>0\right\}
\ \mbox{(for some $\alpha(f)\in(0,1)$ which plays the role of $\alpha(\nu)$ in \eqref{eq:exponents})}$$
and
$$\left\{\frac{S_1(L_f^{(1)}(t))-S_2(L_f^{(2)}(t))}{t^{1-\alpha_1(f)}}\right\}
\ \mbox{(for some $\alpha_1(f)\in(0,1)$ which plays the role of $\alpha_1(\nu)$ in \eqref{eq:exponent})}$$
to some non-degenerating random variables. In particular we expect that, as happens in the previous sections (see eqs.
\eqref{eq:exponents} and \eqref{eq:exponent}), $\alpha(f)=\alpha_1(f)$ when the Lévy process $\{S(t):t\geq 0\}$ is not centered 
(i.e. $m\neq 0$ in Section \ref{sec:gen-results-Beghin-Macci-2022}, or $m_1\neq m_2$ in Section \ref{sec:gen-results-Lee-Macci}),
and $1-\alpha(f)=\frac{1-\alpha_1(f)}{2}$ when the Lévy process $\{S(t):t\geq 0\}$ is centered.

Then one should be able to handle some suitable limits which follow from suitable applications of the G\"artner Ellis Theorem;
more precisely, for every choice of the positive scalings $\{a_t:t>0\}$ such that \eqref{eq:MDconditions} holds, we mean
$$a_t\log\mathbb{E}\left[\exp\left(\kappa_S\left(\frac{\theta}{(a_tt)^{1-\alpha(f)}}\right)L_f(t)\right)\right]$$
for the possible generalization of Proposition \ref{prop:ncMD}, and
$$a_t\left(\log\mathbb{E}\left[\exp\left(\kappa_{S_1}\left(\frac{\theta}{(a_tt)^{1-\alpha_1(f)}}\right)L_f(t)\right)\right]+
\log\mathbb{E}\left[\exp\left(\kappa_{S_2}\left(-\frac{\theta}{(a_tt)^{1-\alpha_1(f)}}\right)L_f(t)\right)\right]\right)$$
for the possible generalization of Proposition \ref{prop:ncMD-bis}.

In our opinion we can prove the generalizations of the other results by requiring some other conditions, and Condition 
\ref{cond:gen-iss} is not enough. We also point out that, if it is possible to prove these generalizations,
the rate functions could not have an explicit expression, and only variational formulas would be available; see 
$I_{\mathrm{MD}}(\cdot;m)$ in Proposition \ref{prop:ncMD} (which can be derived from the variational formula 
\eqref{eq:MD-Legendre-transform}) and $J_{\mathrm{MD}}$ in Proposition \ref{prop:ncMD-bis} (which can be derived from the 
variational formula \eqref{eq:MD-Legendre-transform-bis}).

\section{Comparisons between rate functions}\label{sec:comparisons}
In this section $\{S(t):t\geq 0\}$ is a real-valued Lévy process as in Condition \ref{cond:real-lighttailed-Levyprocess}, with
bounded variation, and it is not a subordinator; thus we can refer to both Conditions \ref{cond:gen-BM} and \ref{cond:gen-LM}, 
and in particular (as stated in Lemma \ref{lem:Chi}) we can refer to the non-trivial independent subordinators 
$\{S_1(t):t\geq 0\}$ and $\{S_2(t):t\geq 0\}$ such that $\{S(t):t\geq 0\}$ is distributed as $\{S_1(t)-S_2(t):t\geq 0\}$.
In particular we can refer to the LDPs in Propositions \ref{prop:LD} and \ref{prop:LD-bis}, which are governed by
the rate functions $I_{\mathrm{LD}}$ and $J_{\mathrm{LD}}$, and to the classes of LDPs in Propositions \ref{prop:ncMD} and 
\ref{prop:ncMD-bis}, which are governed by the rate functions $I_{\mathrm{MD}}(\cdot;m)=I_{\mathrm{MD}}(\cdot;m_1-m_2)$ and 
$J_{\mathrm{MD}}$. All these rate functions uniquely vanish at $x=0$. So, by arguing as in \cite{LeeMacci} (Section 5), we present 
some generalizations of the comparisons between rate functions (at least around $x=0$) presented in that reference. Those comparisons 
allow to compare different convergences to zero; this could be explained by adapting the explanations in \cite{LeeMacci}
(Section 5), and here we omit the details.

\begin{remark}\label{rem:simplified-notation-psi}
	Throughout this section we always assume that $\nu_1=\nu_2=\nu$ for some $\nu\in(0,1)$. So we simply write
	$\Psi_\nu$ in place of the function $\Psi_{\nu_1,\nu_2}$ in Proposition \ref{prop:LD-bis} (see \eqref{eq:LD-GE-limit-type1}).
\end{remark}

We start by comparing $J_{\mathrm{LD}}$ in Proposition \ref{prop:LD-bis} and $I_{\mathrm{LD}}$ in Proposition 
\ref{prop:LD}.

\begin{proposition}\label{prop:comparison-rf-LD}
	Assume that $\nu_1=\nu_2=\nu$ for some $\nu\in(0,1)$. Then $J_{\mathrm{LD}}(0)=I_{\mathrm{LD}}(0)=0$ and, for
	$x\neq 0$, we have $I_{\mathrm{LD}}(x)>J_{\mathrm{LD}}(x)>0$.
\end{proposition}
\begin{proof}
	Firstly, since the range of values of $\Psi_\nu^\prime(\theta)$ (for $\theta$ such that the derivative is well-defined)
	is $(-\infty,\infty)$ (see the final part of the proof of Proposition \ref{prop:LD-bis}, and the change of notation in
	Remark \ref{rem:simplified-notation-psi}), for all $x\in\mathbb{R}$ there exists $\theta_x^{(1)}$ such that 
	$\Psi_\nu^\prime(\theta_x^{(1)})=x$, and therefore
	$$J_{\mathrm{LD}}(x)=\theta_x^{(1)}x-\Psi_\nu(\theta_x^{(1)}).$$
	We recall that $\theta_x^{(1)}=0$ ($\theta_x^{(1)}>0$ and $\theta_x^{(1)}<0$, respectively) if and only if
	$x=0$ ($x>0$ and $x<0$, respectively). Then $J_{\mathrm{LD}}(0)=0$ and, moreover, we have $I_{\mathrm{LD}}(0)=0$; indeed the
	equation $\Lambda_{\nu,S}^\prime(\theta)=0$, i.e.
	$$\frac{1}{\nu}(\kappa_{S_1}(\theta)+\kappa_{S_2}(-\theta))^{1/\nu-1}(\kappa_{S_1}^\prime(\theta)+\kappa_{S_2}^\prime(-\theta))=0,$$
	yields the solution $\theta=0$.
	
	We conclude with the case $x\neq 0$. If $x>0$, then we have
	$$\kappa_{S_1}(\theta_x^{(1)})>\max\{\kappa_{S_1}(\theta_x^{(1)})+\kappa_{S_2}(-\theta_x^{(1)}),0\};$$
	this yields $\Psi_\nu(\theta_x^{(1)})>\Lambda_{\nu,S}(\theta_x^{(1)})$, and therefore
	$$J_{\mathrm{LD}}(x)=\theta_x^{(1)}x-\Psi_\nu(\theta_x^{(1)})<\theta_x^{(1)}x-\Lambda_{\nu,S}(\theta_x^{(1)})
	\leq\sup_{\theta\in\mathbb{R}}\{\theta x-\Lambda_{\nu,S}(\theta)\}=I_{\mathrm{LD}}(x).$$
	Similarly, if $x<0$, then we have
	$$\kappa_{S_2}(-\theta_x^{(1)})>\max\{\kappa_{S_1}(\theta_x^{(1)})+\kappa_{S_2}(-\theta_x^{(1)}),0\};$$
	this yields $\Psi_\nu(\theta_x^{(1)})>\Lambda_{\nu,S}(\theta_x^{(1)})$, and we can conclude following the lines of 
	of the case $x>0$.
\end{proof}

The next Proposition \ref{prop:comparison-rf-MD} provides a similar result which concerns the comparison of $J_{\mathrm{MD}}$
in Proposition \ref{prop:ncMD-bis} and $I_{\mathrm{MD}}(\cdot;m)=I_{\mathrm{MD}}(\cdot;m_1-m_2)$ in Proposition \ref{prop:ncMD}.
In particular we have $I_{\mathrm{MD}}(x;m_1-m_2)>J_{\mathrm{MD}}(x)>0$ for all $x\neq 0$ if and only if $m_1\neq m_2$ (note that, 
if $m_1\neq m_2$, we have $m=m_1-m_2\neq 0$; thus $\alpha(\nu)$ in \eqref{eq:exponents} coincides with $\alpha_1(\nu)$ in 
\eqref{eq:exponent}); on the contrary, if $m_1=m_2$, we have $I_{\mathrm{MD}}(x;0)>J_{\mathrm{MD}}(x)>0$ only if $|x|$ is small 
enough (and strictly positive).

\begin{proposition}\label{prop:comparison-rf-MD}
	We have $J_{\mathrm{MD}}(0)=I_{\mathrm{MD}}(0;m_1-m_2)=0$. Moreover, if $x\neq 0$, we have two cases.
	\begin{enumerate}
		\item If $m_1\neq m_2$, then $I_{\mathrm{MD}}(x;m_1-m_2)>J_{\mathrm{MD}}(x)>0$.
		\item If $m_1=m_2=m_*$ for some $m_*>0$, there exists $\delta_\nu>0$ such that:
		$I_{\mathrm{MD}}(x;0)>J_{\mathrm{MD}}(x)>0$ if $0<|x|<\delta_\nu$,
		$J_{\mathrm{MD}}(x)>I_{\mathrm{MD}}(x;0)>0$ if $|x|>\delta_\nu$,
		and $I_{\mathrm{MD}}(x;0)=J_{\mathrm{MD}}(x)>0$ if $|x|=\delta_\nu$.
	\end{enumerate}
\end{proposition}
\begin{proof}
	The equality $J_{\mathrm{MD}}(0)=I_{\mathrm{MD}}(0;m_1-m_2)=0$ (case $x=0$) is immediate.
	So, in what follows, we take $x\neq 0$. We start with the case $m_1\neq m_2$, and we have two cases.
	\begin{itemize}
		\item Assume that $m_1>m_2$. Then for $x<0$ we have $J_{\mathrm{MD}}(x)<\infty=I_{\mathrm{MD}}(x;m_1-m_2)$.
		For $x>0$ we have $\frac{x}{m_1}<\frac{x}{m_1-m_2}$, which is trivially equivalent to 
		$J_{\mathrm{MD}}(x)<I_{\mathrm{MD}}(x;m_1-m_2)$.
		\item Assume that $m_1<m_2$. Then for $x>0$ we have $J_{\mathrm{MD}}(x)<\infty=I_{\mathrm{MD}}(x;m_1-m_2)$.
		For $x<0$ we have $-\frac{x}{m_2}<-\frac{x}{m_2-m_1}$, which is trivially equivalent to
		$J_{\mathrm{MD}}(x)<I_{\mathrm{MD}}(x;m_1-m_2)$.
	\end{itemize}
	Finally, if $m_1=m_2=m_*$ for some $m_*>0$, the statement to prove trivially holds noting that, for two constants 
	$c_{\nu,m_*}^{(1)},c_{\nu,m_*}^{(2)}>0$, we have $J_{\mathrm{MD}}(x)=c_{\nu,m_*}^{(1)}|x|^{1/(1-\nu)}$ and 
	$I_{\mathrm{MD}}(x;0)=c_{\nu,m_*}^{(2)}|x|^{1/(1-\nu/2)}$.
\end{proof}

We remark that the inequalities around $x=0$ in Propositions \ref{prop:comparison-rf-LD} and \ref{prop:comparison-rf-MD} are not surprising;
indeed we expect to have a slower convergence to zero when we deal with two independent random time-changes (because in that case we have more
randomness).

Now we consider comparisons between rate functions for different values of $\nu\in(0,1)$.
We mean the rate functions in Propositions \ref{prop:LD} and \ref{prop:LD-bis}, and we restrict our attention to a comparison around 
$x=0$. In view of what follows, we consider some slightly different notation: $I_{\mathrm{LD},\nu}$ in place of $I_{\mathrm{LD}}$ in
Proposition \ref{prop:LD}; $J_{\mathrm{LD},\nu}$ in place of $J_{\mathrm{LD}}$ in Proposition \ref{prop:LD-bis}, with $\nu_1=\nu_2=\nu$
for some $\nu\in(0,1)$.

\begin{proposition}\label{prop:comparison-rf-LD-nu}
	Let $\nu,\eta\in(0,1)$ be such that $\eta<\nu$. Then:
	$I_{\mathrm{LD},\eta}(0)=I_{\mathrm{LD},\nu}(0)=0$, $J_{\mathrm{LD},\eta}(0)=J_{\mathrm{LD},\nu}(0)=0$; for some
	$\delta>0$, we have $I_{\mathrm{LD},\eta}(x)>I_{\mathrm{LD},\nu}(x)>0$ and 
	$J_{\mathrm{LD},\eta}(x)>J_{\mathrm{LD},\nu}(x)>0$ for $0<|x|<\delta$.
\end{proposition}
\begin{proof}
	We can say that there exists $\delta>0$ small enough such that, for $|x|<\delta$, there exist $\theta_{x,\nu}^{(1)},\theta_{x,\nu}^{(2)}\in\mathbb{R}$ such that:
	$$I_{\mathrm{LD},\nu}(x)=\theta_{x,\nu}^{(2)}x-\Lambda_{\nu,S}(\theta_{x,\nu}^{(2)})
	\ \mbox{and}\ J_{\mathrm{LD},\nu}(x)=\theta_{x,\nu}^{(1)}x-\Psi_\nu(\theta_{x,\nu}^{(1)});$$
	moreover $\theta_{x,\nu}^{(1)}=\theta_{x,\nu}^{(2)}=0$ if $x=0$, and $\theta_{x,\nu}^{(1)},\theta_{x,\nu}^{(2)}\neq 0$ if 
	$x\neq 0$; finally we have
	$$0\leq\Psi_\nu(\theta_{x,\nu}^{(1)}),\Lambda_{\nu,S}(\theta_{x,\nu}^{(2)})<1.$$
	Then, by taking into account the same formulas with $\eta$ in place of $\nu$ (together with the inequality $\frac{1}{\eta}>\frac{1}{\nu}$),
	it is easy to check that
	$$0\leq\Lambda_{\eta,S}(\theta_{x,\nu}^{(2)})<\Lambda_{\nu,S}(\theta_{x,\nu}^{(2)})<1
	\quad\mbox{and}\quad 0\leq\Psi_\eta(\theta_{x,\nu}^{(1)})<\Psi_\nu(\theta_{x,\nu}^{(1)})<1$$	
	(see \eqref{eq:LD-GE-limit} for the first chain of inequalities, and \eqref{eq:LD-GE-limit-type1} with $\nu_1=\nu_2=\nu$
	for the second chain of inequalities); thus
	$$I_{\mathrm{LD},\nu}(x)=\theta_{x,\nu}^{(2)}x-\Lambda_{\nu,S}(\theta_{x,\nu}^{(2)})
	<\theta_{x,\nu}^{(2)}x-\Lambda_{\eta,S}(\theta_{x,\nu}^{(2)})
	\leq\sup_{\theta\in\mathbb{R}}\{\theta x-\Lambda_{\eta,S}(\theta)\}=I_{\mathrm{LD},\eta}(x)$$
	and
	$$J_{\mathrm{LD},\nu}(x)=\theta_{x,\nu}^{(1)}x-\Psi_\nu(\theta_{x,\nu}^{(1)})
	<\theta_{x,\nu}^{(1)}x-\Psi_\eta(\theta_{x,\nu}^{(1)})
	\leq\sup_{\theta\in\mathbb{R}}\{\theta x-\Psi_\eta(\theta)\}=J_{\mathrm{LD},\eta}(x).$$
	This completes the proof.
\end{proof}

As a consequence of Proposition \ref{prop:comparison-rf-LD-nu} we can say that, in all the above convergences 
to zero governed by some rate function (that uniquely vanishes at zero), the smaller the $\nu$, the faster the convergence of
the random variables to zero.

We also remark that we can obtain a version of Proposition \ref{prop:comparison-rf-LD-nu} in terms of the rate functions
$I_{\mathrm{MD}}(\cdot;m_1-m_2)=I_{\mathrm{MD},\nu}(\cdot;m_1-m_2)$ in Proposition \ref{prop:ncMD} and 
$J_{\mathrm{MD}}=J_{\mathrm{MD},\nu}$ in Proposition \ref{prop:ncMD-bis}. Indeed we can obtain the same kind of inequalities, 
and this is easy to check because we have explicit expressions of the rate functions (here we omit the details).

\section{An example with independent tempered stable subordinators}\label{sec:diff-TSS}
Throughout this section we consider the examples presented below.

\begin{example}\label{ex:diff-TSS}
	Let $\{S_1(t):t\geq 0\}$ and $\{S_2(t):t\geq 0\}$ be two independent tempered stable subordinators
	with parameters $\beta\in(0,1)$ and $r_1>0$, and $\beta\in(0,1)$ and $r_2>0$, respectively. Then
	$$\kappa_{S_i}(\theta):=\left\{\begin{array}{ll}
		r_i^\beta-(r_i-\theta)^\beta&\ \mbox{if}\ \theta\leq r_i\\
		\infty&\ \mbox{if}\ \theta> r_i
	\end{array}\right.$$
	for $i=1,2$; thus
	$$\kappa_S(\theta):=\kappa_{S_1}(\theta)+\kappa_{S_2}(-\theta)=\left\{\begin{array}{ll}
		r_1^\beta-(r_1-\theta)^\beta+r_2^\beta-(r_2+\theta)^\beta&\ \mbox{if}\ -r_2\leq\theta\leq r_1\\
		\infty&\ \mbox{otherwise}.
	\end{array}\right.$$
\end{example}

Note that, in order to be consistent in the comparisons between $I_{\mathrm{LD}}$ and $J_{\mathrm{LD}}$, we always take
the rate function $J_{\mathrm{LD}}$ in Proposition \ref{prop:LD-bis} with $\nu_1=\nu_2=\nu$. Moreover, we remark that, 
for Example \ref{ex:diff-TSS}, the function $\Lambda_{\nu,S}$ is essentially smooth because
$$\lim_{\theta\downarrow -r_2}\Lambda_{\nu,S}^\prime(\theta)=-\infty\quad\mbox{and}\quad
\lim_{\theta\uparrow r_1}\Lambda_{\nu,S}^\prime(\theta)=+\infty;$$
indeed these conditions holds if and only if $\kappa_S(r_1)$ and $\kappa_S(-r_2)$ are positive, and in fact we have
$$\kappa_S(r_1)=\kappa_S(-r_2)=r_1^\beta+r_2^\beta-(r_1+r_2)^\beta
=(r_1+r_2)^\beta\left(\left(\frac{r_1}{r_1+r_2}\right)^\beta+\left(\frac{r_2}{r_1+r_2}\right)^\beta-1\right)>0.$$

We start with Figure \ref{fig:different-values-of-nu}. The graphs agree with the inequalities in Proposition 
\ref{prop:comparison-rf-LD-nu}. Moreover Figure \ref{fig:different-values-of-nu} is more informative than Figure 3 in
\cite{LeeMacci}; indeed the graphs of $J_{\mathrm{LD}}=J_{\mathrm{LD},\nu}$ on the right (for different values of $\nu$)
show that the inequalities in Proposition \ref{prop:comparison-rf-LD-nu} hold only in a neighborhood of the origin $x=0$.

\begin{figure}[ht]
	\centering
	\includegraphics[scale=0.5]{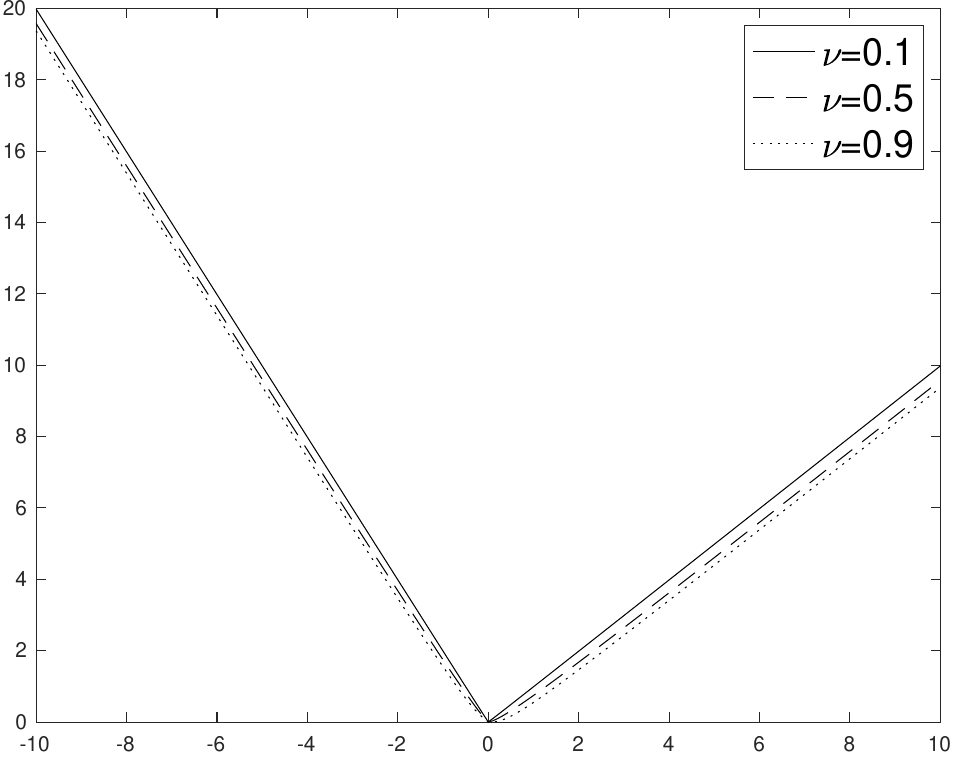}
	\includegraphics[scale=0.5]{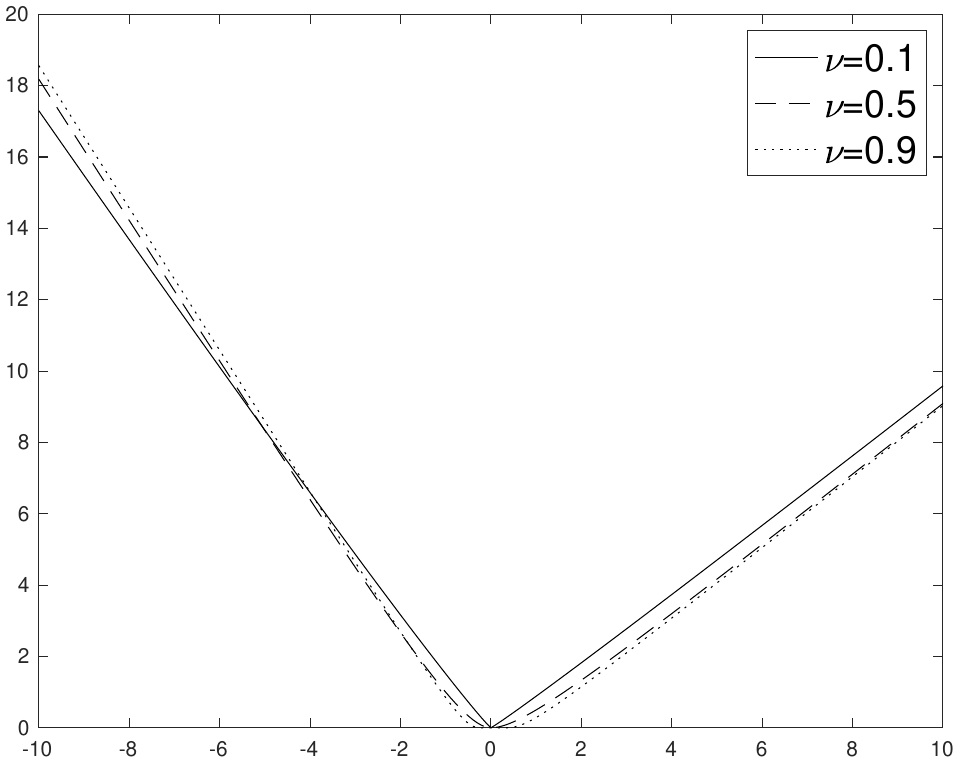}
	\caption{Rate functions $I_{\mathrm{LD}}=I_{\mathrm{LD},\nu}$ (on the left) and $J_{\mathrm{LD}}=J_{\mathrm{LD},\nu}$ 
	(on the right) for the processes in Example \ref{ex:diff-TSS} and different values of $\nu$ ($\nu=0.1,0.5,0.9$). 
	Numerical values of the other parameters: $r_1=1$, $r_2=2$, $\beta=0.5$.}
	\label{fig:different-values-of-nu}
\end{figure}

In the next Figures \ref{fig:different-values-of-beta} and \ref{fig:different-values-of-r2} we take different values of
$\beta$ and of $r_2$, respectively, when the other paramaters are fixed. The graphs in these figures agree with 
Proposition \ref{prop:comparison-rf-LD}.

\begin{figure}[ht]
	\centering
	\includegraphics[scale=0.33]{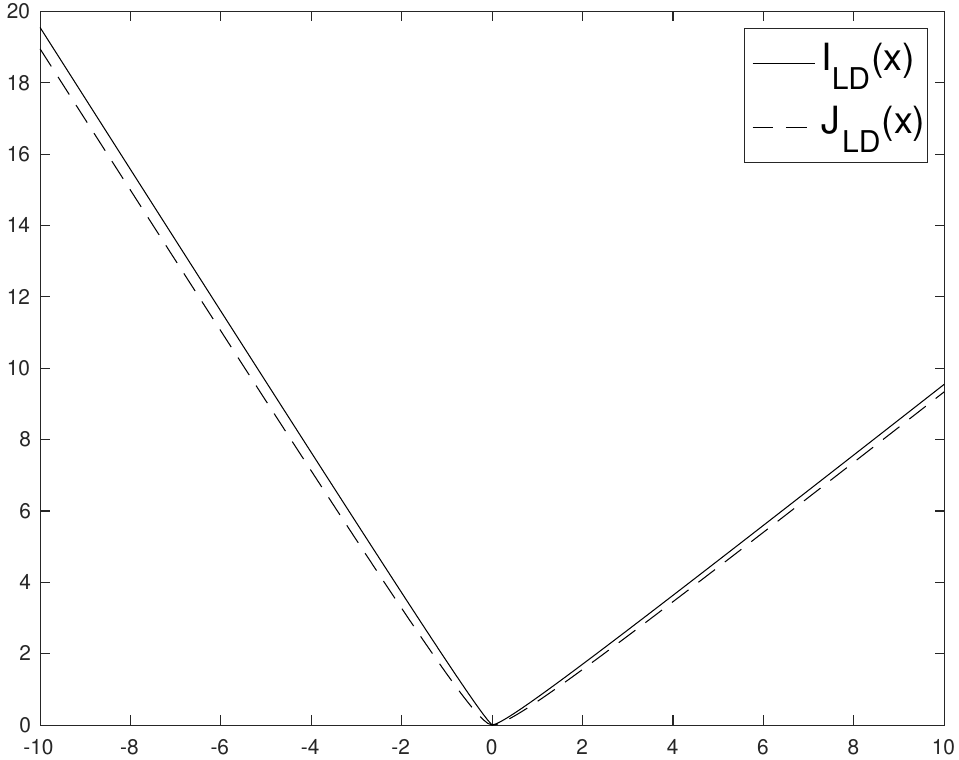}
	\includegraphics[scale=0.33]{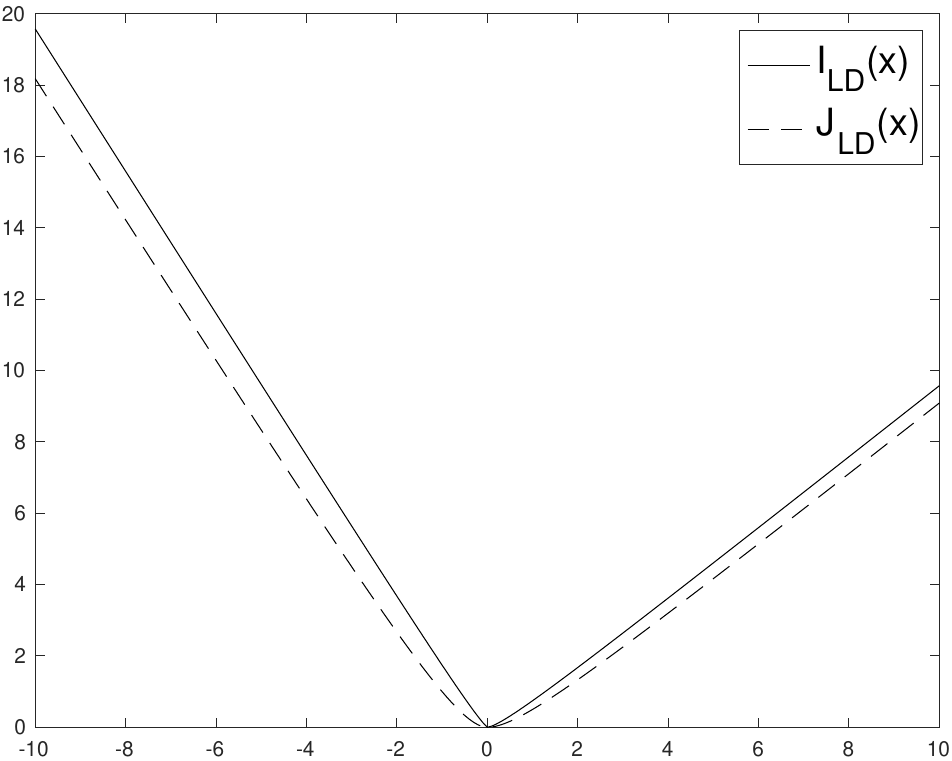}
	\includegraphics[scale=0.33]{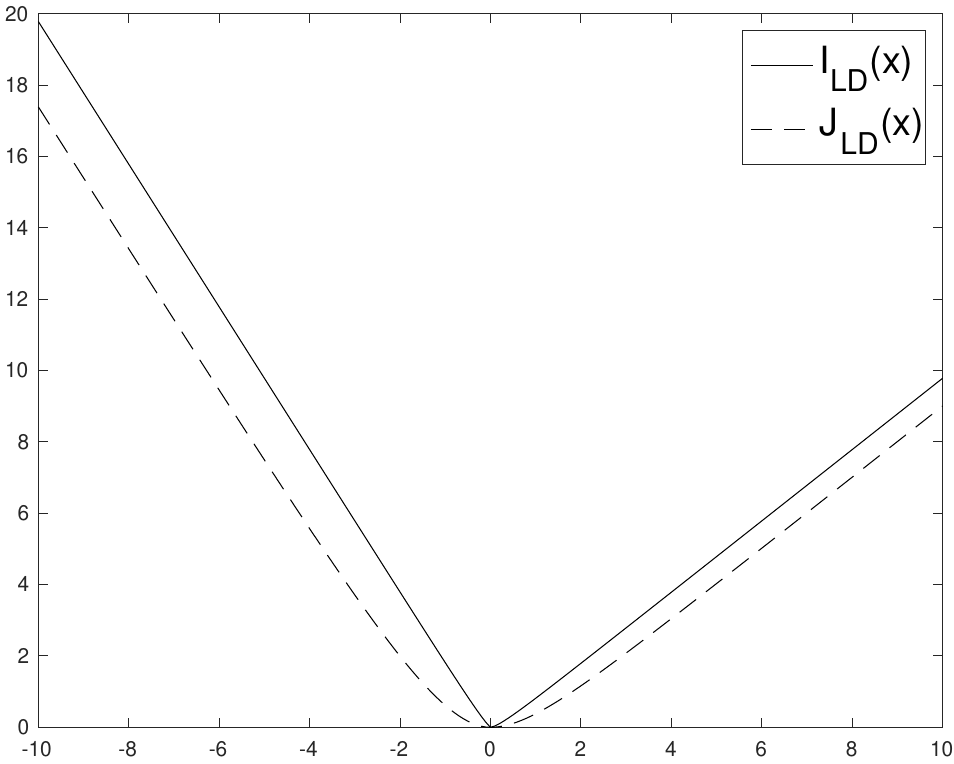}
	\caption{Rate functions $I_{\mathrm{LD}}$ and $J_{\mathrm{LD}}$ for the processes in Example \ref{ex:diff-TSS} and 
	different values of $\beta$ ($\beta=0.3,0.5,0.7$ from left to right). Numerical values of the other parameters: 
	$r_1=1$, $r_2=2$, $\nu=0.5$.}
	\label{fig:different-values-of-beta}
\end{figure}

\begin{figure}[ht]
	\centering
	\includegraphics[scale=0.33]{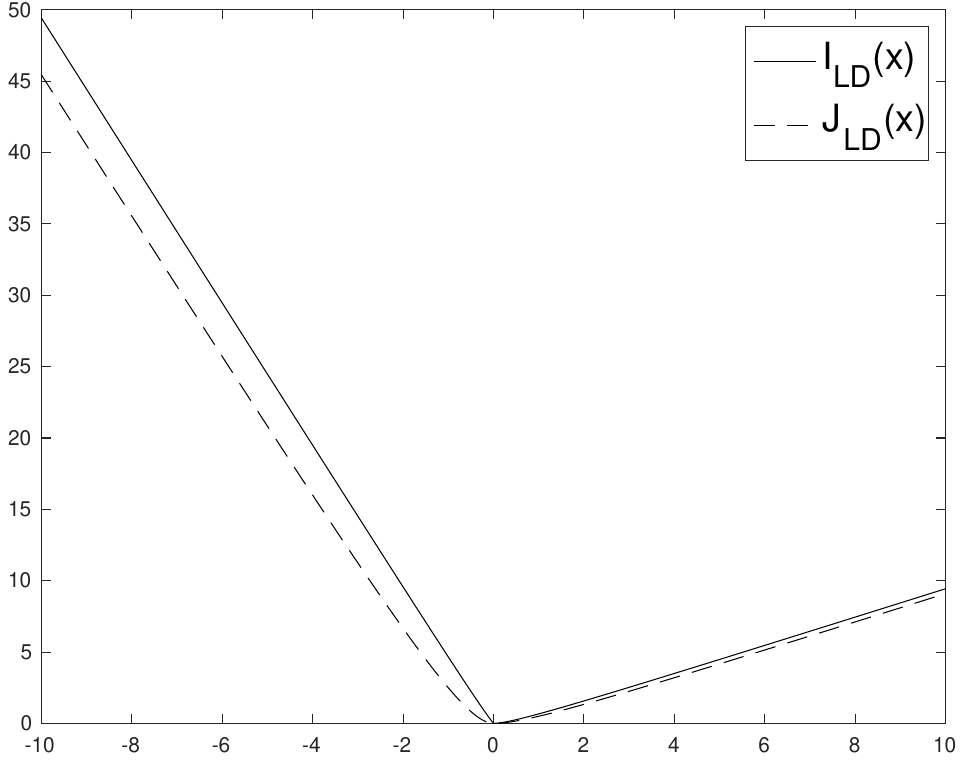}
	\includegraphics[scale=0.33]{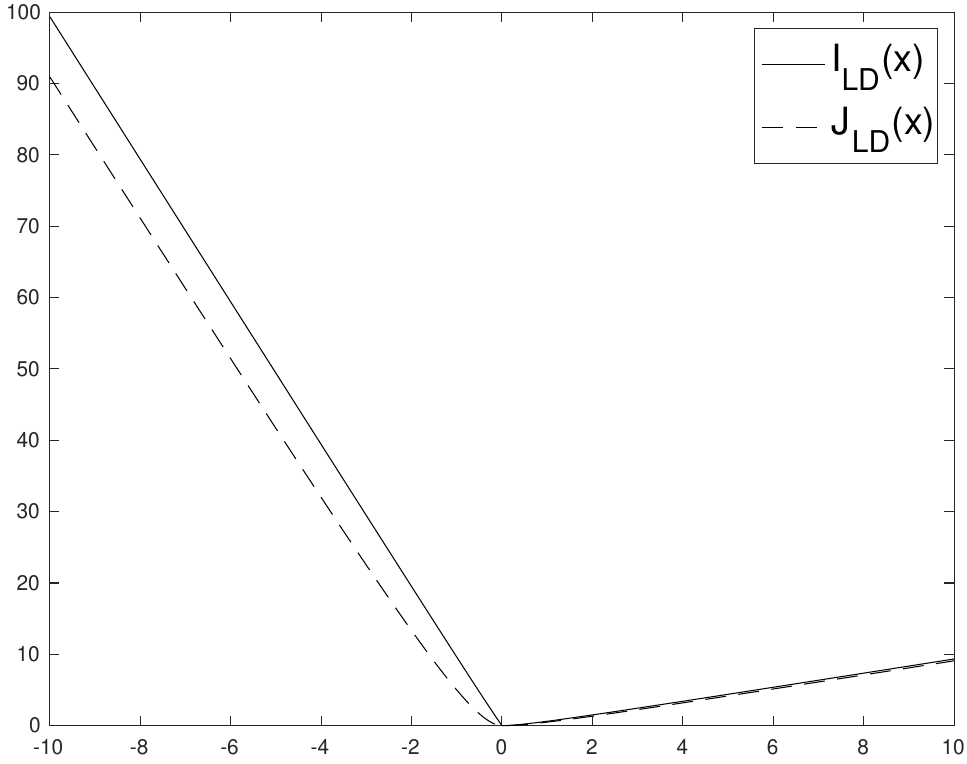}
	\includegraphics[scale=0.33]{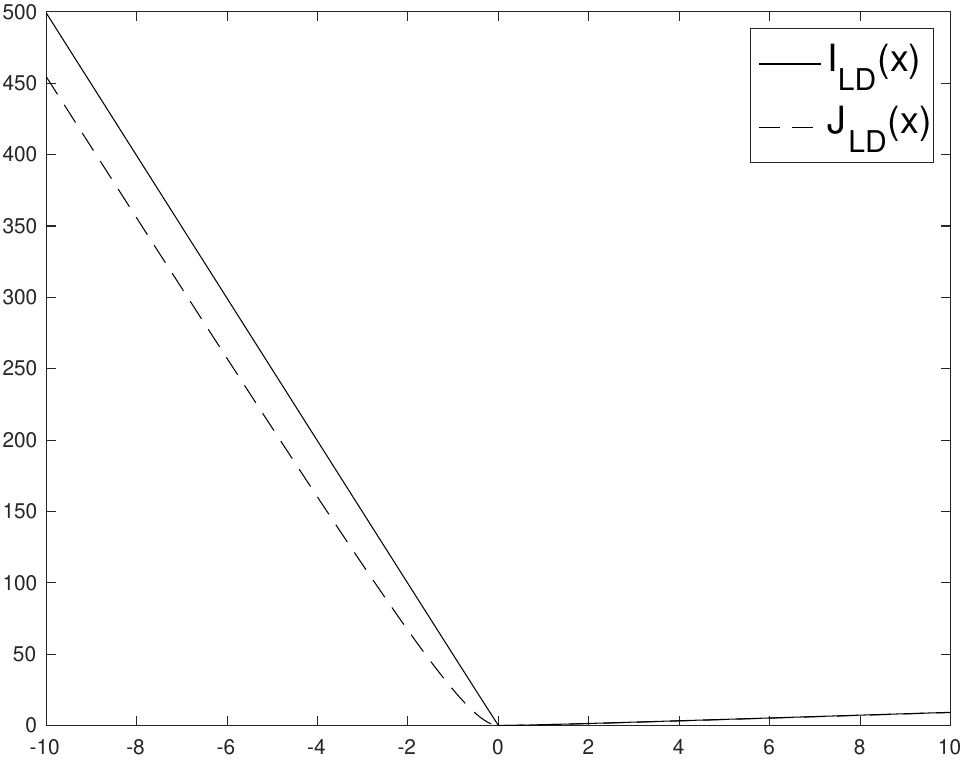}
	\caption{Rate functions $I_{\mathrm{LD}}$ and $J_{\mathrm{LD}}$ for the processes in Example \ref{ex:diff-TSS} and 
	different values of $r_2$ ($r_2=5,10,50$ from left to right). Numerical values of the other 
	parameters: $r_1=1$, $\nu=0.5$, $\beta=0.5$.}
    \label{fig:different-values-of-r2}
\end{figure}

\paragraph{Acknowledgements.}
We thank Prof. Zhiyi Chi for some comments on the proof of Lemma \ref{lem:Chi}. We also thank a referee for some useful
comments which led to the addition of Section \ref{sec:referee} and an improvement in the presentation of the paper.

\paragraph{Funding.} 
A.I. acknowledges the support of MUR-PRIN 2022 PNRR (project P2022XSF5H ‘‘Sto\-chastic Models in Biomathematics 
and Applications’’) and by INdAM-GNCS.\\
C.M. acknowledges the support of MUR Excellence Department Project awarded to the Department of Mathematics, 
University of Rome Tor Vergata (CUP E83C23000330006), by University of Rome Tor Vergata (project "Asymptotic Properties 
in Probability" (CUP E83C22001780005)) and by INdAM-GNAMPA.\\
A.M. acknowledges the support of MUR-PRIN 2022 (project 2022XZSAFN ‘‘Anomalous Phenomena on Regular and Irregular Domains:
Approximating Complexity for the Applied Sciences’’), by MUR-PRIN 2022 PNRR (project P2022XSF5H ‘‘Stochastic Models in 
Biomathematics and Applications’’) and by INdAM-GNCS.

\end{document}